\documentclass{amsart}

\usepackage{palatino} 
\usepackage{pinlabel}
\usepackage{pdfpages}
\usepackage{hyperref, multicol}
\usepackage{tikz-cd}

\usepackage{amsmath,amssymb,amsthm,verbatim, amsfonts,amscd,flafter,epsf, epsfig,graphicx,verbatim,pinlabel,mathrsfs}
\usepackage[all]{xy}
\usepackage{epsf}
\usepackage[abs]{overpic}
\usepackage{epstopdf}
\usepackage{color}
\usepackage{mathtools}
\usepackage{amsthm}

\definecolor{red}{rgb}{1,0,0}
\definecolor{blue}{rgb}{0,0,1}
\definecolor{green}{rgb}{0,1,0}

\newtheorem{theorem}{Theorem}[section]
\newtheorem{lemma}[theorem]{Lemma}

\newtheorem{corollary}[theorem]{Corollary}
\newtheorem{proposition}[theorem]{Proposition}
\newtheorem{question}[theorem]{Question}

\theoremstyle{definition}
\newtheorem{definition}[theorem]{Definition}

\newtheorem{remark}[theorem]{Remark}

\newtheorem{example}[theorem]{Example}

\def\ws{\widetilde{\Sigma}}
\def\wpts{\widetilde{\mathbf{p}}}

\newtheorem*{liftingthm1}{Theorem~\ref{thm:premain}}
\newtheorem*{liftingthm2}{Theorem~\ref{thm:main}}
\newtheorem*{homomthm}{Theorem~\ref{thm:homom2}}
\newtheorem*{coverthm}{Theorem~\ref{thm:cover}}

\title{Liftable braids and the coloured braid groupoid}

\author[Joan Licata]{Joan Licata}
\address{Mathematical Sciences Institute, Australian National University  \& International Research Lab, France-Australia Mathematical Sciences and Interactions}

\email{joan.licata@anu.edu.au}

\author[Vera V\'ertesi]{Vera V\'ertesi}
\address{University of Vienna}
\email{vera.vertesi@univie.ac.at}

\begin{document}

\maketitle

\begin{abstract} When $\pi:\ws\rightarrow D^2$ is a cover of the disc branched over $n$ marked points, the braid group $B_n$ acts on the disc by homeomorphisms fixing the marked points setwise. A  braid $\beta$ \textit{lifts} if there is a homeomorphism $\widetilde{\beta}\in \textit{Mod}(\ws)$ such that $\beta\circ \pi=\pi\circ \widetilde{\beta}$.  For arbitrary covers, the  \textit{lifting homomorphism} taking $\beta$ to $\widetilde{\beta}$ is only defined on a proper subgroup  of the braid group.  This paper extends the lifting homomorphism to a map from a coloured braid groupoid to a mapping class groupoid for all simple covers of the disc.  We characterise the lift of every coloured braid, recovering the classical lifting homomorphism on the liftable braid group.  \end{abstract}

\section{Introduction}

This paper focuses on a specific case of the following general question: when one surface covers another, how are their mapping class groups related?  Birman and Hilden famously took this perspective in their discovery of a presentation for the mapping class group of the genus two surface, and further exploration has exposed deep and beautiful connections between algebraic and topological structures \cite{BH1}.  Like these, we consider branched covers of the disc and identify the mapping class group of the marked disc with the braid group.    

This relationship is well understood for double branched covers: every braid acting on the branch values in the disc lifts to a homeomorphism of the double branched cover of the disc.  In more general covers, however, the set of \textit{liftable braids} forms a proper subgroup of the braid group that acts on a marked disc.  We focus on $d$-fold \textit{simple covers}; these are the branched covers with minimal branching, in the sense that no component of the branch locus has branching degree greater than $2$ (Definition~\ref{def:simple}).

\begin{figure}[h]
\begin{center}
\includegraphics[scale=.8]{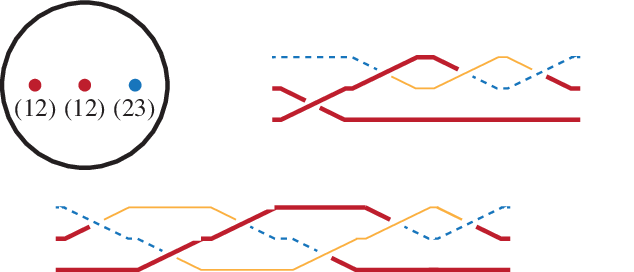}
\caption{ The labels on the marked disc determine a $3$-fold simple branched cover.  The (equivalent) coloured braids act on this marked and labled disc.  Each distinct strand colour is identified with a distinct transposition $(12), (23)$, or $(13)$.   } 
\label{fig:ex1}
\end{center}
\end{figure}

A $d$-fold simple branched cover of the disc is determined by labeling each branch value by a transposition in the symmetric group $S_d$, so the natural mapping classes that act on the marked disc are braids whose strands are coloured by transpositions.  As strands cross in a diagram representing the braid, the label on the overstrand is preserved, while the label on the understrand is conjugated by the overstrand label.  See Figure~\ref{fig:ex1} for an example of a marked and labeled disc and two coloured braids that act on it.

A coloured braid lifts to a mapping class of the branched cover if and only if the initial and terminal strand labels agree.   Figure~\ref{fig:ex1} shows two diagrams for a liftable coloured braid. This lifting criterion is clear, but it is an existence statement rather than a constructive one; this paper addresses how to identify the lift of an arbitrary braid.

Given a branched covering of the marked disc $\pi:\ws\rightarrow  D^2$, we decorate $\ws$ with a system of indexed arcs and assign arcslides and index swaps to each coloured braid generator. When a liftable coloured braid $\beta=\sigma_{i_m}\dots \sigma_{i_1}$ acts on the branch values in the base, we apply the corresponding operations to the decorations on $\ws$.  

\begin{liftingthm1} There is a unique mapping class $\Phi(\beta)$ that takes the image of the decorated $\ws$ under the action of $\beta$  to $\ws$ with the original decorations.  This mapping class $\Phi(\beta)$ is the lift of $\beta$.
\end{liftingthm1}

\begin{figure}[h]
\begin{center}
\includegraphics[scale=.7]{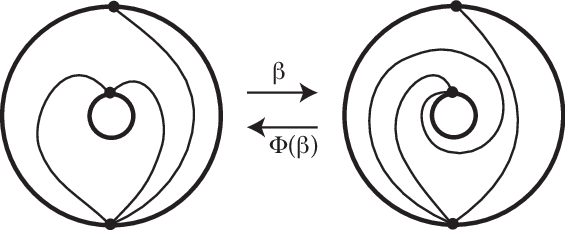}
\caption{ The braid $\beta$ shown in Figure~\ref{fig:ex1} induces arcslides that transform the decorations on the left annulus to the decorations on the right annulus.  The mapping class $\Phi(\beta)$ taking the right-hand annulus to the left-hand annulus is the lift of $\beta$.} 
\label{fig:introex}
\end{center}
\end{figure}

Figure~\ref{fig:introex} shows a pair of decorated annuli that cover the disc from Figure~\ref{fig:ex1}.  The operations assigned to the braid $\beta$ in Figure~\ref{fig:ex1} transform the decorations on the left to the decorations on the right.   It is straightforward to see that the annulus on the right maps to the annulus on the left via a positive Dehn twist; this is the lift of $\beta$.

The algorithm producing the  two decorated surfaces is combinatorial and works equally well for any factorisation of the original braid $\beta$.  This is an improvement over existing lifting results, which depended factoring $\beta$ as a product of generators of the liftable braid group.  To achieve this, we replace the liftable subgroup of the coloured braid group $B_n$ by the coloured braid groupoid $\text{ColB}_{n,d}$; this allows us to study coloured braids whose initial and terminal labels differ.  We then  generalise the map that takes liftable braids to mapping classes and instead lift the entire coloured braid groupoid to a mapping class groupoid.

In fact, we associate a pair of groupoids to each equivalence class of branched covers of the disc. The first is a \textit{mapping class groupoid} $\mathcal{M}$ whose objects are branched covers of the disc and whose morphisms are homeomorphisms between them.   We also define a \textit{graphical groupoid} $\mathcal{G}$ whose objects are branched covers decorated with the systems of curves  noted above; see Definition~\ref{def:graphobj} for the precise definition.  Morphisms in $\mathcal{G}$ are explicit combinatorial operations on the curves, and Definition~\ref{def:homom} assigns a morphism in $\mathcal{G}$ to each coloured braid.  

Applications of mapping class groups often rely on characterising a mapping class by its action on some auxiliary curves, and in this spirit, pairs of objects in $\mathcal{G}$  give rise to morphisms in $\mathcal{M}$.  When a liftable braid $\beta$ acts on an object $O\in \mathcal{G}$, we have already introduced the notation $\Phi(\beta)$ for the homeomorphism mapping $\beta\cdot O$ to $O$.  More generally, Proposition~\ref{prop:label} extends this to a function $\Phi:\text{ColB}_{n,d}\rightarrow \mathcal{M}$.

\begin{homomthm} The map $\Phi:\text{ColB}_{n,d}\rightarrow \mathcal{M}$ is a groupoid homomorphism. 
\end{homomthm}

This homomorphism lifts the entire coloured braid groupoid:
\begin{liftingthm2}  

Suppose that $\pi_1:\ws_1\rightarrow D^2_1$ and $\pi_2:\ws_2\rightarrow D^2_2$ are branched covers of the marked and labeled disc with $\beta: D^2_1\rightarrow D^2_2$ a coloured braid.  Then \[\pi_2\circ\Phi(\beta)=\beta\circ \pi_1.\] \end{liftingthm2}

In the case of a double branched cover, each elementary braid generator lifts to a Dehn twist in the branched cover, and one may consider various ways to generalise this important example.  Wajnryb coined the term ``geometric lift'' for any homomorphism with this property and asked if there are non-geometric maps from braid groups to mapping class groups \cite{Waj1}. Szepietowksi established the existence of non-geometric lifts, with other constructions soon following \cite{Szep} \cite{GM} \cite{BodTill}.  Our generalisation from group to groupoid has a different flavour, but $\Phi$ is nevertheless geometric in the sense that it takes each liftable generator $\sigma_i$ to a Dehn twist, while still meaningfully interpreting the lifts of other generators.  

Composition in a groupoid is more subtle than in a group, but the lifting problem is simpler, as the complexity of lifting an arbitrary braid reduces to the complexity of composing lifts of simple generators. There exist arbitrarily long concatenations of coloured braid generators which cannot be truncated to form a liftable braid, so it is impractical to study lifts of braids to mapping classes via fixed braid diagrams.  In contrast, any braid diagram decomposes as a concatenation of coloured braids that are liftable to the mapping class groupoid. 

Lifting the coloured braid groupoid to a mapping class groupoid offers a further computational advantage.  Existing work, culminating in a result of Wajnryb-Wi\'sniowska-Wajryb, establishes a finite generating set for the liftable braid group and articulates the lift of each \cite{MP}, \cite{WWW} \cite{Apol}.  However, there is no algorithm for expressing a liftable braid in terms of these generators, making lifts difficult to compute in practice.  For example, the top braid diagram in Figure~\ref{fig:ex1} factors as a product of two liftable braids: (1) a crossing between two strands with the same label and (2) three half-twists between strands with different labels.  In contrast, the equivalent braid diagram shown on below does not decompose as a concatenation of diagrams for liftable braids.  Nevertheless, Example~\ref{ex:ex2} computes the lift of this braid from the factorisation given by the lower diagram, as each generator is liftable with respect to the mapping class groupoid.

In the final section we construct a pair of $2$-dimensional CW-complexes $X_\mathcal{G}$ and $X_\mathcal{M}$ from the groupoids $\mathcal{G}$ and $\mathcal{M}$, respectively.

\begin{coverthm} For $\ws$ the branched cover of the marked disc labeled by $\tau_0$, we have \[ \pi_1(X_\mathcal{M}, \tau_0)\cong\text{Mod}(\ws).\]  Furthermore, the CW-complex $X_\mathcal{G}$ is the universal cover of $X_\mathcal{M}$.  
\end{coverthm}   

\subsection{Acknowledgements}
This material is based upon work supported by the National Science Foundation under Grant No. DMS-1439786 while the authors were in residence at the Institute for Computational and Experimental Research in Mathematics in Providence, RI, during the Braids program. The first author would like to thank Rima Chatterjee for interesting conversations about lifting braids to simple covers. This research was supported in part by the Austrian Science Fund (FWF) P 34318. For open access purposes, the author has applied a CC BY public copyright license to any author-accepted manuscript version arising from this submission.

\section{Background and Definitions}\label{sec:background}

This paper relates two distinct mathematical objects, branched covers and coloured braids.  We provide a brief introduction to each here and refer the reader to \cite{MP} or \cite{Apol} for more details.    

\subsection{Branched covers}\label{sec:brcov}

The following definition of a branched cover will suffice for our purposes, although more general versions exist in the literature. 

 \begin{definition}\label{def:simple}
The map $\pi:(\widetilde{\Sigma}, \wpts) \rightarrow (\Sigma, \mathbf{p})$ is a \textit{branched cover} if $\wpts\subset \widetilde{\Sigma}$ is a codimension-2 submanifold  and the restriction $\underline{\pi}:(\ws\setminus \wpts )\rightarrow (\Sigma\setminus \textbf{p})$ is a covering map.   A branched cover is a \textit{$d$-fold simple cover} if if $|\pi^{-1}(x)|\in \{d, d-1\}$ for every $x\in \Sigma$.  The subset of $\pi^{-1}(\mathbf{p})$ where branching occurs is called the \textit{branch locus}. 
\end{definition}

Branched covers are classified up to \textit{equivalence}, where $\pi:\widetilde{\Sigma}\rightarrow \Sigma$ and $\pi':\widetilde{\Sigma}'\rightarrow \Sigma'$ are equivalent if there exist $\beta\in \text{Homeo}^+(\Sigma, \Sigma')$ and $\widetilde{\beta}\in \text{Homeo}^+(\ws, \ws')$ such that $\pi'\circ \widetilde{\beta}=\beta\circ \pi$. In this paper, we study only one equivalence class of branched covers at a time.  

When $\Sigma$ is a surface, the branching locus consists of \textit{branch points} and  we call  $\mathbf{p}$ the \textit{branch values}.  In this case, the local model for $s$-fold branching takes a neighbourhood of a branch point to a neighbourhood of a branch value via a map locally parametrised by $z\mapsto z^s$. 

A branched cover of a surface with boundary is determined up to equivalence by data describing the local behaviour of the cover over each branch value \cite{BE} \cite{MP}.  Specifically, fix a  basepoint $x$ on the boundary and a collection of curves with disjoint interiors that connect $x$ to the branch values.   Following \cite{WWW}, we call this set of curves a \textit{basis}. To each curve $\gamma_i$ in a basis, associate the clockwise-oriented loop based at $x$ that encircles $\gamma_i$ and no other branch values.  Let $S_d$ denote the symmetric group on $d$ letters.
The (unbranched) covering map $\underline{\pi}:(\ws\setminus \wpts )\rightarrow (\Sigma\setminus \textbf{p})$ has a monodromy homomorphism $\pi_1(\Sigma \setminus \mathbf{p}, x)\rightarrow S_d$,  defined up to conjugation in $S_d$, and we label each branch value with the permutation associated to the lift of the loop encircling it.  When $\pi$ is a simple branched cover, each monodromy label is a single transposition $t_i\in S_d$.   

Henceforth we will consider only simple branched covers of the disc with degree at least 3.  In this case, define the \textit{standard basis} to be the one shown in black on Figure~\ref{fig:construct}: the basepoint is placed at the bottom of the disc and the curves are line segments connecting it to branch values lying on the horizontal equator. 

\begin{figure}[h]
\begin{center}
\includegraphics[scale=.7]{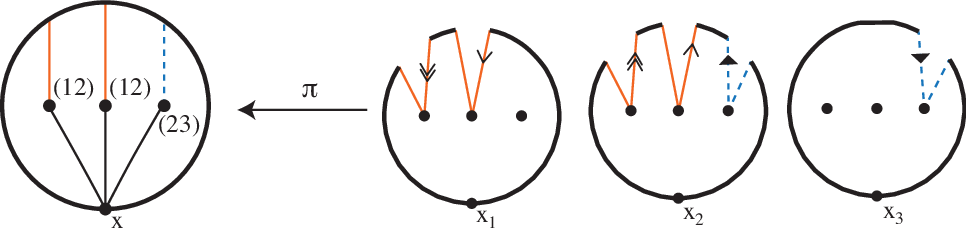}
\caption{ Left: a disc with the standard basis $\Gamma$, vertical cuts $\Delta$, and monodromy labels $t_i$ determining a simple cover. Right: the branched cover associated to the initial data is an identification space built from three copies of the original disc cut open along $\Delta$ arcs.} \label{fig:construct}
\end{center}
\end{figure}

The  transpositions labeling a marked disc  provide concrete instructions for building simple branched covers, as shown in Figure~\ref{fig:construct}.  First, cut  $D^2$ along  a set of disjoint arcs $\Delta$ with the property that each $p_i$ is connected to $\partial \Sigma$ by a single arc $\delta_i$ whose interior is disjoint from $\Gamma$.  It follows that $\delta_i$ intersects the loop encircling $\gamma_i$ once.  Take $d$ copies of this cut-open disc and label the lifts of the basepoint $x_1, x_2, \dots, x_d$.  The boundary of the cut-open surface has distinguished intervals, each containing a preimage of a single branch value. If the branch value $p_i$  is labeled by the transposition $(jk)$, identify the distinguished intervals on the $j^{th}$ and $k^{th}$ copies of the cut open disc and re-glue along the cut in the other copies.  It follows that the loop based at $x\in D^2$ that encircles $\gamma_i$ lifts to a path in $\ws$ connecting the basepoint lifts $x_j$ and $x_k$.

\begin{remark}\label{rmk:build} This observation leads to  a convenient depiction of a simple cover of the disc as the neighbourhood of a \textit{rigid vertex graph}.  Also called a \textit{ribbon graph}, such graphs fix the cyclic order of the edges at each vertex.  In the present case,  the vertex set is the lifted basepoints $\{x_i\}$ and there is an edge through the lifts of the cuts $\Delta$ that contain branch points.  If a branch value $p_i$ is labeled with the transposition $(jk)$, the edge through $\widetilde{p_i}$ has endpoints at $x_j$ and $x_k$.  Strictly speaking, the edges respect a linear order rather than a cylic one, as they  meet each basepoint $x_j$ with the same order the original basis entered $x$; this condition that will reappear in Definition~\ref{def:graphobj}.    

An example is shown in Figure~\ref{fig:graph}. We will sometimes draw covering surfaces thus for convenience, but we emphasise that the cover is completely determined by the copies of the cut-open disc and the identifications along the cuts as in Figure~\ref{fig:construct}.

\begin{figure}[h]
\begin{center}
\includegraphics[scale=.7]{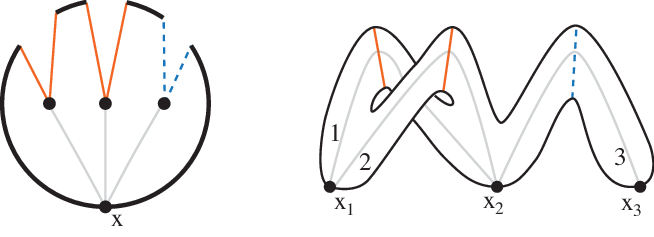}
\caption{ Left: the cut-open disc from Figure~\ref{fig:construct}, shown with $\Gamma$ arcs. Right: the branched cover is a neighbourhood of the  $\pi^{-1}(\Gamma)$ arcs that terminate at branch points.} 
\label{fig:graph}
\end{center}
\end{figure}
\end{remark}

The \textit{total monodromy} $\mu$ of a simple cover is the image of the boundary of $D^2$ under the monodromy homomorphism.  Cycles in the total monodromy are in one-to-one correspondence with boundary components of $\widetilde{\Sigma}$. In the construction described above, each boundary component contains the basepoints $x_j$ with indices in the corresponding cycle, and $\mu$ can be computed as a product of the transpositions $t_1\dots t_n$.  This data suffices to classify covers:

\begin{theorem}[\cite{MP}, \cite{BE}] Two simple branched covers of the disc are equivalent if and only if their total monodromies are conjugate in $S_d$.
\end{theorem}

Let $\mathcal{T}(\mu)$ denote the set of labels on a marked disc associated to a fixed total monodromy $\mu$; when $\mu$ is understood, we will often write simply $\mathcal{T}$.  As shown by Mulazzani and Piergallini, for each equivalence class of covers there is a distinguished total monodromy $\mu$ and label $\tau_0 \in \mathcal{T}(\mu)$ with the following form:
\[ (12) ((12)) (23) ((23)) \dots (d-1\   d) \dots (d-1 \ d).\]

Here, the double parentheses indicate that the second copy of the transposition may or may not appear, so there are one or two copies of each transposition $(i \ i+1)$ for $1\leq i\leq d-2$ and an odd number of copies of the final transposition $(d-1\ d)$.

\subsection{The coloured braid groupoid} 

The $n$-strand braid group $B_n$ is ubiquitous in mathematics. We identify $B_n$ with the mapping class group of the disc with $n$ marked points, where mapping classes fix the boundary of the disc pointwise and the set of marked points setwise.  Let $\sigma_i$ denote the element that exhanges the $i^{th}$ and $(i+1)^{th}$ marked point via a counterclockwise rotation supported in a neighbourhood of the interval connecting them.  The $k-1$ maps of this form generate $B_n$ and give rise to the classical Artin presentation:
 
\begin{equation}\label{def:braidgroup} B_n\cong \langle \sigma_1, \dots, \sigma_{n-1} \ | \ \sigma_i\sigma_{i+1}\sigma_i=\sigma_{i+1}\sigma_i\sigma_{i+1}, \sigma_{i}\sigma_j=\sigma_{j}\sigma_i \text{ for  }|i-j|>1\rangle.\end{equation}

We also view braids as tangles in $D^2\times I$ where the strands are nowhere tangent to $D^2\times \{t\}$;  braids are drawn from left to right, following the conventions shown in Figure~\ref{fig:pos} to unite these two perspectives.

\begin{figure}[h]
\begin{center}
\includegraphics[scale=.6]{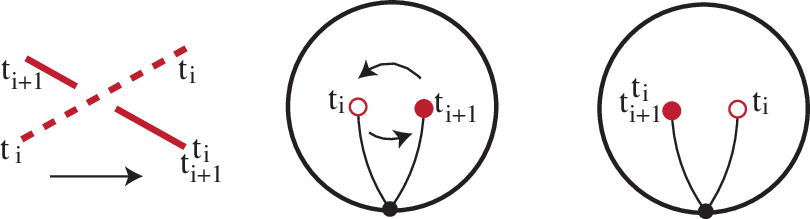}
\caption{Conventions for a positive generator $\sigma_i$. Here $t_i$ is the initial transposition labeling strand $i$ and $t_{i+1}$ is the initial transposition labeling strand $i+1$.} 
\label{fig:pos}
\end{center}
\end{figure}

The $n$-strand $d$-coloured braid groupoid $\text{ColB}_{n,d}$ enhances $B_n$ by decorating each strand of a braid diagram with a permutation in $S_d$ subject to the following rule also noted in the introduction: when a strand labeled $t_i$ crosses  over a strand labeled $t_j$, then the label on the overstrand is preserved as $t_i$ while the understrand label changes to the conjugate $t_j^{t_i}=t_it_jt_i$.  It follows that once initial labels in a coloured braid are specified, the remaining labels are completely determined.  For economy, we will sometimes make use of this, but  we more frequently choose to over-label braid diagrams for clarity.  See, for example, Figure~\ref{fig:ex1}.

Each coloured braid comes with a pair of distinguished end labels, and two coloured braids are composable via concatenation if and only the initial labels of the second braid match the terminal labels of the first.  Because all braids are invertible but not all pairs of braids are composable, coloured braids form a groupoid rather than a group.  Following \cite{Apol}, an object of the $n$-strand $d$-coloured braid groupoid is a copy of the marked disc $D$ where each of the $n$ marked points is labeled  with a transposition in $S_d$.  The morphisms are coloured braids on $n$ strands, up to isotopy fixing the endpoints, and $\beta$ lies in Hom($D_1, D_2$) when the initial strand labels match those of $D_1$ and the final labels match $D_2$. 

Just as a group can be viewed as either a category with one object or a set where any two elements may be multiplied, we may view a groupoid as a category where every morphism is invertible or as a set where only some pairs of elements may be composed.  Although equivalent, these perspectives prompt different narrative treatments and we will make use of both. 

  \subsection{Braids and branched covers}\label{sec:braids}
  
  Given a  disc marked with $n$ points  $\mathbf{p}$, any braid $\beta\in B_{n}$ acts on $(D^2, \mathbf{p})$ as a mapping class.  When the marked points have monodromy labels relative to some basis, then it is natural to view $\beta$ as a coloured braid that acts on the labels by the \textit{Hurwitz action}. Suppose that $\tau=(t_1,t_2, \dots, t_n)$.  Then
  \[\sigma_i\cdot \tau=\sigma_i\cdot(t_1, t_2, \dots, t_i, t_{i+1}, \dots t_n)=(t_1, t_2, \dots, t_i t_{i+1} t_i, t_{i}, \dots t_n).\]  
     This extends to an action of any braid $\beta$ by factoring $\beta$ as a product of generators and applying the map associated to each one in turn.  
     
     Equivalently, one may start with a coloured braid.  In this case, the marked points inherit labels $\tau$ from the initial strand labels and we set $\pi_\tau:\ws_\tau \rightarrow D$ to be the associated cover, constructed as usual with respect to the standard basis. 
     
In either case, the coloured braid maps between states of the coloured braid groupoid.  When the initial and final states agree, the same surface $\ws_\tau$ covers both labeled discs.

\begin{definition} The coloured braid $\beta$ is \textit{liftable} with respect to $\pi:\ws\rightarrow D$ if there exists a mapping class $\widetilde{\beta}: \ws\rightarrow \ws$ such that $\beta\cdot \pi=\pi\cdot \widetilde{\beta}$.   
\end{definition}

In fact, every coloured braid with the same initial and terminal colours is liftable, and the lift is unique \cite{MP}.  Liftable coloured braids are therefore composable, and the subgroup of coloured braids that are liftable with respect to $\pi$ is denoted by $LB_{n,d}$.   As noted above, we will assume throughout that $d\geq 3$.  The map taking $\beta$ to $\widetilde{\beta}$ is the \textit{lifting homomorphism} from $LB_{n,d}\rightarrow  \text{Mod}(\ws)$. 

The  liftable braid group has a number of nice properties.  For example, suppose $C$ is an arc embedded in $D$ with $C\cap \mathbf{p}=\partial C$.  A \textit{twist along $C$} exchanges the endpoints of $\partial C$ via a counterclockwise rotation of a neighbourhood of $C$ in the complement of the other branch values.  The twist along $C$ is denoted by $\sigma_C$, and each Artin generator is a twist along the line segment between consecutive branch values $p_i$ and $p_{i+1}$.   In \cite{MP}, Mulazzani and Piergallini showed that for every $C$, some power of $\sigma_C$ is liftable, and furthermore, that $LB_{n,d}$ is generated by liftable powers of these twists.  In their language, an arc $C$ is Type $i$ if the $i^{th}$ power of $\sigma_C$ is the minimal liftable power, and every arc is Type 1, 2, or 3.  

Generating sets were described more precisely for specific families of branched covers, culminating in the work of Wajnryb and Wi\'sniokska-Wajnryb that gave an explicit  finite set of arcs $\mathcal{W}$ on the marked disc labeled by $\tau_0$ whose liftable twists generate $\text{BCol}_{n,d}$ \cite{WWW}  \cite{WW1}.

In the next section, we lift the coloured braid groupoid associated to a fixed total monodromy to the mapping class groupoid of the covering surface.

\section{The mapping class groupoid}\label{sec:groupoids}

The existence of a lifting functor from the coloured braid groupoid to a mapping class groupoid of the covering surface is known; see e.g., \cite{Apol}.  While the mapping class group of a surface has a precise definition, there is no distinguished mapping class groupoid, as the assignment of distinct objects to a fixed surface is a matter of choice.   In this section, we choose a model for the mapping class groupoid of a covering surface and define a homomorphism from the coloured braid groupoid to this model.

Given $n$ and $d$, the coloured braid groupoid splits into connected components indexed by the total monodromy $\mu$ of the associated covers.  When we refer to ``the'' coloured braid groupoid, we will always mean a fixed connected component, as distinct components have no interesting topological or algebraic interactions. Throughout this section, let $\mathcal{T}=\mathcal{T}(\mu)$ be the set of distinct labelings $\tau=(t_1, t_2, \dots, t_n)$ such that $t_1t_2 \dots t_n=\mu$.  For $\tau\in \mathcal{T}$, let $(\ws_\tau, \pi_\tau)$ be the associated branched covering constructed as an identification space as in Section~\ref{sec:brcov}.  Fix lifts of the basepoint $\{x_1, \dots, x_d\}\subset \partial \ws_\tau$; we consider the labeled basepoints to be an intrinsic feature of the covering surface.  For all $\tau \in \mathcal{T}$, the surfaces $\ws_\tau$ are homeomorphic via maps respecting these numbered basepoints.

\begin{definition}\label{def:mcgoid} Set $\mathcal{M}(\mu)$ to be the groupoid whose objects are the covering surfaces $\ws_{\tau}$ for $\tau\in \mathcal{T}$, considered up to isotopy fixing $\partial \ws_{\tau}$. The morphisms of $\mathcal{M}(\mu)$ are homeomorphisms that preserve the indexed basepoints.  Two morphisms acting on the same object are equivalent if their images are isotopic relative to $\partial \ws_\tau$. 
\end{definition}

If all the points are labeled with the same transposition, then there is a single object and $\mathcal{M}(\mu)$ is the mapping class group of the double branched cover of the marked disc.  Under the assumption that $d\geq 3$, a connected component of the coloured braid groupoid will have multiple labels and hence, multiple objects.   For each object $\ws_\tau$, restricting to the self-morphisms of $\ws_\tau$ recovers a copy of the mapping class group of $\ws_\tau$, and a classically liftable braid acting on the disc labeled by $\tau$ is one whose lift is such a mapping class.  However, using a groupoid allows a more general definition:

\begin{definition}\label{def:genlift} Let $\beta$ be a coloured braid acting on the marked disc with initial labels $\tau$.  The \textit{lift} of $\beta$ is the homeomorphism $\widetilde{\beta}: \ws_{\tau}\rightarrow \ws_{\beta\cdot \tau}$ such that \[\beta\circ \pi_\tau=\pi_{\beta\cdot \tau}\circ \widetilde{\beta}.\]
\end{definition}

When $\beta\cdot \tau=\tau$, this recovers the classical definition of the lift of a coloured braid.  However, every coloured braid is liftable in the sense of Definition~\ref{def:genlift}; see Theorem~\ref{thm:lift}.  

In order to identify the lift of an arbitrary coloured braid, we will define a second structure, the graphical groupoid $\mathcal{G}(\mu)$.  Each object in $\mathcal{G}(\mu)$ is a pair $(\ws_\tau, \widetilde{\Gamma})$ consisting of an object from $\mathcal{M}(\mu)$ together with some decorations $\widetilde{\Gamma}$ which are characterised in Definition~\ref{def:graphobj}.  For each label $\tau\in \mathcal{T}$, there is a distinguished object $O_\tau=(\ws_\tau, \widetilde{\Gamma}_\tau)$. Morphisms in $\mathcal{G}(\mu)$ preserve the surface but alter the decorations. 
Definition~\ref{def:homom} defines a groupoid homomorphism $\phi: \text{ColB}_{n,d}\rightarrow \mathcal{G}(\mu)$.

\begin{proposition}\label{prop:label} If $\beta$ is a coloured braid that acts on the marked disc labeled by $\tau$, then there is a homeomorphism $\Phi(\beta)\colon\widetilde{\Sigma}_\tau\to \widetilde{\Sigma}_{\beta\cdot\tau}$ such that $\Phi(\beta)(\phi(\beta) \cdot O_\tau)=O_{\beta\cdot \tau}$. The isotopy class of $\Phi(\beta)$ relative to the boundary is unique.
\end{proposition}

When $\beta\cdot \tau=\tau$, it follows that $\Phi(\beta)$ is a map from $\ws_\tau$ to itself.

\begin{theorem}\label{thm:premain} When $\beta$ is a liftable braid,  the mapping class $\Phi(\beta)$ is the lift $\widetilde{\beta}$.
\end{theorem} 

A stronger version of this result is stated as Theorem~\ref{thm:main}.

The value in introducing a second groupoid and a second homeomorphism may not be immediately apparent, but $\mathcal{G}(\mu)$ acts as a combinatorial proxy for an abstract groupoid of surfaces and mapping classes.  The next two sections precisely define the structures noted here, and a number of examples may be found in Section~\ref{sec:ex}.

\subsection{The graphical groupoid}\label{sec:gobj}

Choose a surface $\ws_\tau$ for some $\tau\in\mathcal{T}$.

\begin{definition}\label{def:graphobj}  A $(\mu, n, d)$ \textit{graphical object} $O$ is a copy of $\ws_\tau$ for some $\tau\in \mathcal{T}$ decorated with indexed curves $\widetilde\Gamma=\{\widetilde{\gamma}_j\}$ for $1\leq j\leq n$ as follows:
\begin{enumerate}
\item\label{cond1} each $\widetilde{\gamma}_j$ has endpoints on distinct $x_i$ and their interiors are disjointly embedded;
\item each $x_i$ is an endpoint of at least one $\widetilde{\gamma}_j$;
\item\label{disc} the curves $\widetilde\Gamma$ cut $\ws$ into $n$ discs, each of which contains one interval on $\partial \ws_\tau$;
\item\label{cond4} when multiple $\widetilde{\gamma}_j$ have an endpoint on a fixed $x_i$, the linear order in which they meet $x_i$ agrees with the linear order of their indices.
\end{enumerate}

Each object is defined up to isotopy fixing  $\partial \ws_\tau$.
\end{definition}

See Figure~\ref{fig:goex} for two examples of graphical objects.  Note that by (\ref{disc}), $\ws_\tau$ deformation retracts onto $\widetilde{\Gamma}$.

\begin{figure}[h]
\begin{center}
\includegraphics[scale=.6]{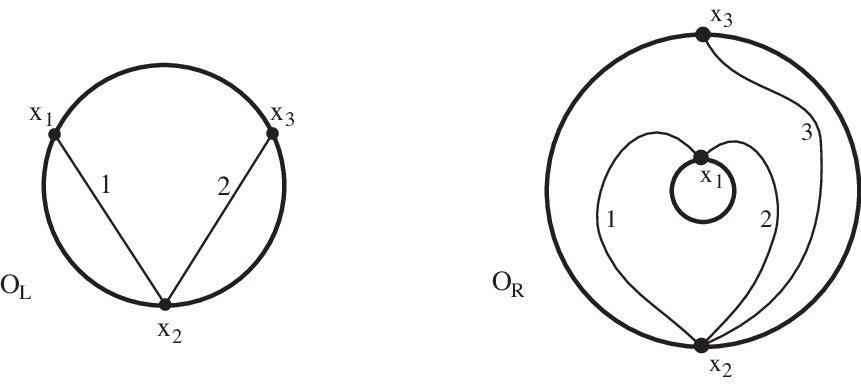}
\caption{ Left: $O_L$ is a $\big( (123), 2, 3\big)$ graphical object.  Right: $O_R$ is a $\big((23), 3, 3\big)$ graphical object. } 
\label{fig:goex}
\end{center}
\end{figure}

One may associate an $n$-tuple of labels in $\mathcal{T}$ to any $(\mu, n, d)$ graphical object $O=(\ws_\rho, \widetilde{\Gamma})$ by assigning the transposition $t_i=(jk)$ to the branch value $p_i$, where $x_j$ and $x_k$ are the endpoints of $\widetilde{\gamma}_i$ on $O$.  Write $L(O)=\tau$ when $\tau$ is the $n$-tuple of labels associated thus to $O$, noting that this $\tau$ is independent of the label $\rho$ associated to the underlying covering surface.  In Figure~\ref{fig:goex}, $L(O_L)= \big(  (12), (23)\big)$ and $L(O_R)= \big(  (12), (12), (23)\big)$.

\begin{definition}\label{def:gomaps}  Given a graphical object $O=(\ws_\rho, \widetilde{\Gamma})$ with $L(O)=\tau$, let $\widetilde{\sigma}^\tau_i\cdot \widetilde{\Gamma}$ be the set of curves on $\ws_\rho$ constructed from $\widetilde{\Gamma}$ as follows: 
at each shared endpoint of $\widetilde{\gamma}_i$ and $\widetilde{\gamma}_{i+1}$, arcslide $\widetilde{\gamma}_{i+1}$ along $\widetilde{\gamma}_{i}$.  For any configuration of  $\widetilde{\gamma}_i$ and $\widetilde{\gamma}_{i+1}$, exchange their indices after any arcslides have been performed. 

We also write $\widetilde{\sigma}^\tau_i\cdot O=(\ws_\rho, \widetilde{\sigma}^\tau_i\cdot \widetilde{\Gamma})$, reflecting the fact that the surface is fixed while the decorations change.
\end{definition}

\begin{figure}[h]
\begin{center}
\includegraphics[scale=.8]{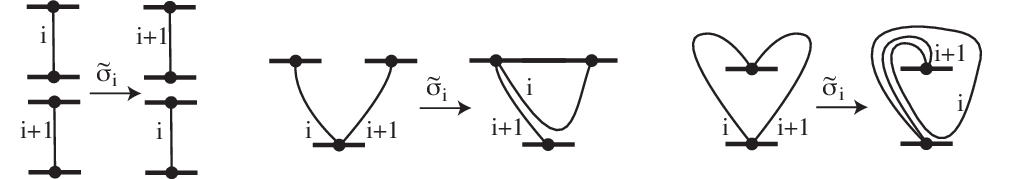}
\caption{  Depending on the configuration of $\widetilde{\Gamma}$ curves on a graphical object, the move labeled $\widetilde{\sigma}_i$ takes one of the forms above.
 } 
\label{fig:moves}  
\end{center}
\end{figure}

These moves are best understood by reference to Figure~\ref{fig:moves}. Note that there are three combinatorially distinct cases, depending on how many endpoints of $\widetilde{\gamma}_{i}$ and $\widetilde{\gamma}_{i+1}$ are shared.

One may easily check that the curves $\widetilde{\sigma}^\tau_i\cdot \widetilde{\Gamma}$ satisfy the hypotheses of Definition~\ref{def:graphobj}, so each $\widetilde{\sigma}^\tau_i$ is in fact a map between graphical objects associated to the same surface.   For convenience, we use the same symbol $\widetilde{\sigma}^\tau_i$ to denote the distinct functions that act on graphical objects associated to the same $\tau$,  as this introduces no ambiguity in practice. 

\begin{definition}\label{def:grgr} The \textit{graphical groupoid}  is the groupoid of graphical objects: \[\mathcal{G}(\mu)=\coprod_{\tau\in \mathcal{T}} (\ws_\tau, \widetilde{\Gamma}).\]  Morphisms are generated by the maps $\widetilde{\sigma}^\tau_i$, subject to the relations that two morphisms are equivalent if they have the same domain and range.  
\end{definition}

The reader may wish to glance ahead at Figures~\ref{fig:s3} and \ref{fig:morph} for examples of composing morphisms in the graphical groupoid to alter the decorations on $\ws_\tau$.  

The final result in this section hints at the motivation for the notation.

\begin{proposition}\label{prop:comp} 
Let $\sigma_i$ be a coloured Artin generator acting on the disc labeled with $\tau$ and let $O \in \mathcal{G}(\mu)$ satisfy $L(O)=\tau$. Then $L(\widetilde{\sigma}^\tau_i\cdot O)=\sigma_i\cdot \tau$.  
\end{proposition}

\begin{proof} Recall the formula for the Hurwitz map given in Section~\ref{sec:braids}:
  \[\sigma_i\cdot(t_1, t_2, \dots, t_i, t_{i+1}, \dots t_n)=(t_1, t_2, \dots, t_i t_{i+1} t_i, t_{i}, \dots t_n).\]  
  
 In each of the three geometrically distinct versions of $\sigma_i^\tau$, we compare the change in labels above to the new endpoints of $\widetilde{\gamma}_i$ and $\widetilde{\gamma}_{i+1}$.

\begin{enumerate} 
\item If the original $\widetilde{\gamma}_i$ and $\widetilde{\gamma}_{i+1}$ were disjoint, then the new $\widetilde{\gamma}_{i+1}$ curve connects the same basepoints as the original $\widetilde{\gamma}_i$ curve.  This agrees with the Hurwitz map, as conjugating one transposition by a disjoint transposition is the identity map.  
\item If the original $\widetilde{\gamma}_i$ and $\widetilde{\gamma}_{i+1}$ shared two endpoints, then they share the same two endpoints after the arcslides are applied.  This agrees with the Hurwitz map, as conjugating a transposition by itself is the identity map.  
\item If the original $\widetilde{\gamma}_i$ and $\widetilde{\gamma}_{i+1}$ shared a single endpoint $x_k$, then sliding $\widetilde{\gamma}_{i+1}$ along $\widetilde{\gamma}_i$ replaces the original label $(jk)$ by $(jl)$, where $x_l$ was the other endpoint of $\widetilde{\gamma}_i$.  This agrees with the Hurwitz conjugation map: $(kl)(jk)(kl)=(jl)$. 
\end{enumerate}
\end{proof}

\subsection{Branched covers and the graphical category}

To each $\tau \in\mathcal{T}$, we associate  a distinguished graphical object $O_\tau\in \mathcal{G}(\mu)$, as follows.

Lift each basis curve $\gamma_i$ on the marked disc to its preimage on $\ws_\tau$.   Two of these preimages meet at a branch point; denote the union of these two arcs by $\widetilde{\gamma_i}$ and discard the components that terminate at non-branching preimages of $p_i$.   Define $\widetilde{\Gamma}_\tau:=\cup_i \widetilde{\gamma_i}$ and set $O_\tau=(\ws_\tau, \widetilde{\Gamma}_\tau)$.  Note that $\widetilde{\Gamma}_\tau$ is the graph described in Remark~\ref{rmk:build}.  The right-hand side of Figure~\ref{fig:graph} provides an example of this construction.  

It is not difficult to verify that $O_\tau$ is a graphical object, with the most subtle point being the claim that $\widetilde{\Gamma}_\tau$ cuts $\ws_\tau$ into discs; to verify this, cut $d$ copies of $D$ along both the basis arcs $\Gamma$ and the cut arcs $\Delta$ before making the identifications along $\Delta$. 

This construction immediately implies the following:

\begin{corollary} With $O_\tau$ as above, $L(O_\tau)=\tau$.  
\end{corollary}

Recall that for each fixed graphical object $O$, there is a morphism $\widetilde{\sigma}^{L(O)}_i$ that acts on $O$ by arcslides and index exchanges on $\widetilde{\gamma}_i$ and $\widetilde{\gamma}_{i+1}$.

It follows from Proposition~\ref{prop:comp} that the map sending the Artin generator $\sigma_i$ to the morphism $\widetilde{\sigma}_i$ respects composition.  Specifically, recall that the coloured braid $\sigma_{i_2}\sigma_{i_1}$ acts on the marked disc  with initial labels $\tau$ exactly when the coloured Artin generator $\sigma_{i_1}$ acts on $\tau$ and the coloured Artin generator $\sigma_{i_2}$ acts on $\sigma_{i_1}\cdot \tau$.  Proposition~\ref{prop:comp} implies that if the morphism $\widetilde{\sigma}^\tau_{i_1}$ acts on a graphical object $O$, then $\widetilde{\sigma}^{\sigma_{i_1}\cdot \tau}_{i_2}$ will act on $\widetilde{\sigma}_{i_1}^\tau\cdot O$.  

The notation quickly becomes unwieldy, so we henceforth drop the superscript  identifying the label of the graphical object,  writing $\widetilde{\sigma}_i \cdot O$ instead of $\widetilde{\sigma}_i^\tau \cdot O$.  This is analogous to writing $\sigma_i$ to denote the coloured Artin generator that acts on the marked disc with labels $\tau$;  similarly, it causes no confusion in practice.

\begin{definition}\label{def:homom} For any coloured braid \[\beta=\sigma^{\epsilon_m}_{i_m}\dots\sigma^{\epsilon_1}_{i_1}\] factored as a product of coloured Artin generators with $\epsilon_i\in \pm1$, define \[\phi(\beta)=\widetilde{\sigma}^{\epsilon_m}_{i_m}\dots \widetilde{\sigma}^{\epsilon_1}_{i_1}.\] 
\end{definition} 

If $\beta$ acts on a disc with labels $\tau$, then $\phi(\beta)$ acts on any graphical objects $O$ with $L(O)=\tau$.  In fact, the map $\phi$ is a homomorphism from the coloured braid groupoid to the graphical groupoid in the following sense:

\begin{proposition}\label{prop:braidrel} Suppose that $\beta=\sigma_{i_m}\dots\sigma_{i_1}$ and $\beta=\sigma_{j_p}\dots\sigma_{j_1}$ are two factorisations of $\beta$.  Then $\widetilde{\sigma}_{i_m}\dots \widetilde{\sigma}_{i_1}=\widetilde{\sigma}_{j_p}\dots \widetilde{\sigma}_{j_1}$.
\end{proposition}

\begin{proof}[Proof of Proposition~\ref{prop:braidrel}] To show that $\phi(\beta)$ is independent of the factorisation of $\beta$, we must verify that relations in the coloured braid groupoid map to relations in the mapping class groupoid: 
\begin{enumerate}
\item Far Pairs: ${\widetilde{\sigma}}_i{\widetilde{\sigma}}_{j}={\widetilde{\sigma}}_{j}{\widetilde{\sigma}}_i$ for $|i-j|>2$; and

\item Braid Relation: $\widetilde{\sigma}_i\widetilde\sigma_{i+1}\widetilde{\sigma}_i=\widetilde\sigma_{i+1}\widetilde{\sigma}_i\widetilde\sigma_{i+1}$.
\end{enumerate}

Each of these breaks into a number of distinct cases depending on the relationship of the $\widetilde{\gamma}_k$ curves involved.  We enumerate the cases, but show only a few arguments in detail.

\begin{figure}[h]
\begin{center}
\includegraphics[scale=.6]{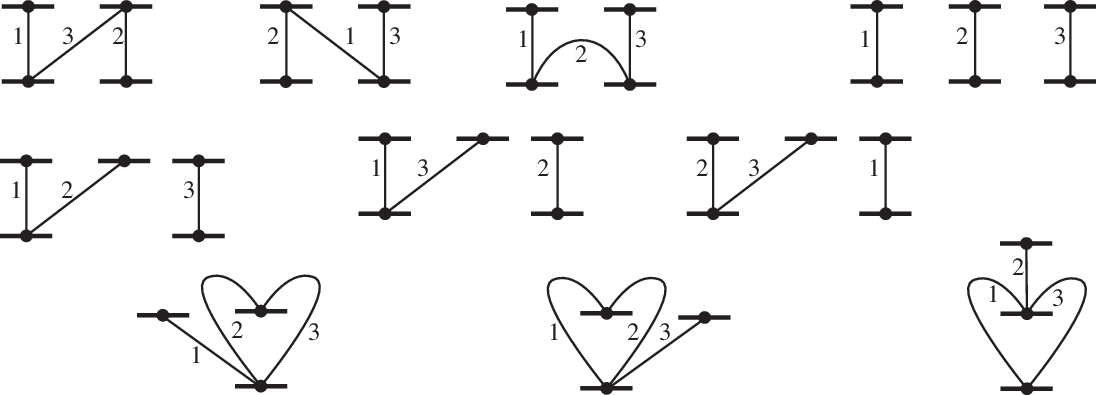}
\caption{The maps $\widetilde{\sigma}_1\widetilde\sigma_{2}\widetilde{\sigma}_1$ and $\widetilde\sigma_{2}\widetilde{\sigma}_1\widetilde\sigma_{2}$ fix the graphical object away from a neighbourhood of $\widetilde{\gamma}_1$,  $\widetilde{\gamma}_2$, and $\widetilde{\gamma}_3$. This figure shows all possible configurations of $\widetilde{\gamma}_1$,  $\widetilde{\gamma}_2$, and $\widetilde{\gamma}_3$, up to homeomorphism. } 
\label{fig:cases}
\end{center}
\end{figure}

\textbf{Far Pairs:} Fix a graphical class and the corresponding $\widetilde{\sigma}_i\widetilde\sigma_{j}$ and $\widetilde{\sigma}_j\widetilde\sigma_{i}$ maps that act on it.  We will examine how these change a representative graphical object in the class.  Each composed map is the identity away from four distinct $\Gamma$ curves: $\widetilde{\gamma}_i, \widetilde{\gamma}_{i+1}, \widetilde{\gamma}_j,\widetilde{\gamma}_{j+1}$. 

 The change to a graphical object under the action of $\widetilde{\sigma}_i$  is supported in a neighbourhood of $\widetilde{\gamma}_i \cup \widetilde{\gamma}_{i+1}$.  It almost follows that $\widetilde{\sigma}_i$ and $\widetilde{\sigma}_j$ alter disjoint subsets of the graphical object, but  the affected $\widetilde{\Gamma}$ curves may meet at some basepoints.  Nevertheless, the ordering condition at the points $\{x_i\}$ ensures that the resulting graphical objects agree there, so in fact, the relation  holds.

\textbf{Braid Relation:} We study the braid relation similarly, fixing a graphical class and compositions that act on it.  For concreteness, consider $\widetilde{\sigma}_1\widetilde\sigma_{2}\widetilde{\sigma}_1$ and $\widetilde\sigma_{2}\widetilde{\sigma}_1\widetilde\sigma_{2}$, so that a representative graphical object is fixed away from a neighbourhood of $\widetilde{\gamma}_1$,  $\widetilde{\gamma}_2$, and $\widetilde{\gamma}_3$.   The combinatorially distinct configurations in which these three curves can meet are shown in Figure~\ref{fig:cases}.

Figure~\ref{fig:cases2} demonstrates that the braid relation holds in one case, and the others follow similarly.

\begin{figure}[h]
\begin{center}
\includegraphics[scale=.7]{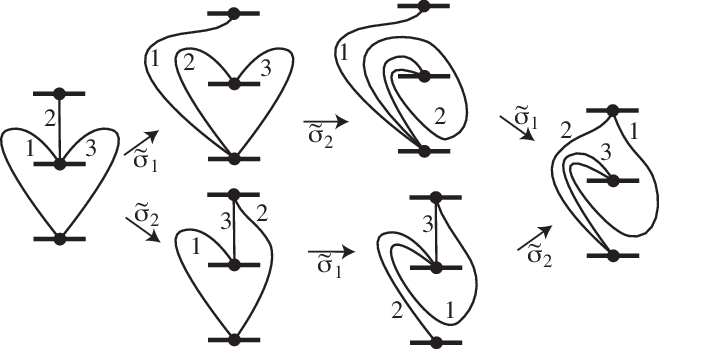}
\caption{The maps $\widetilde{\sigma}_1\widetilde\sigma_{2}\widetilde{\sigma}_1$ and $\widetilde\sigma_{2}\widetilde{\sigma}_1\widetilde\sigma_{2}$ act identically on the initial graphical object.} 
\label{fig:cases2}
\end{center}
\end{figure}
\end{proof}

Having mapped coloured braids to  morphisms in the graphical groupoid brings us significantly closer to explicitly identifying the lift of a braid.  Observe that  any graphical object is homeomorphic to some canonical graphical object:

\begin{lemma}\label{lem:lab} Let $O\in \mathcal{G}(\mu)$ be a graphical object and suppose $L(O)=\tau$.  Then there is a unique isotopy class of homeomorphisms that sends $O$ to $O_\tau$.  
\end{lemma}

The lemma implies Proposition~\ref{prop:label}.

\begin{proof} Given a graphical object $(\ws_\rho, \widetilde\Gamma)$, the surface $\ws_\rho$ deformation retracts onto the rigid vertex graph $\widetilde{\Gamma}$.  We have $L\big(\ws_\rho, \widetilde{\Gamma}\big)=\tau$ exactly when $\widetilde{\Gamma}$ connects the same labeled basepoints as $\widetilde{\Gamma}_\tau$, and this identification extends to a homeomorphism of the initial $\ws_\rho$ onto $\ws_\tau$.

 The claim that this mapping class is unique is an application of the Alexander Method; see, for example,  \cite[Section 2.3]{FM}.
\end{proof}

\begin{definition}\label{def:Phi} If $\beta$ is a coloured braid that acts on the marked disc labeled by $\tau$, define  $\Phi(\beta)$ to be the homeomorphism taking $\phi(\beta) \cdot O_\tau$ to $O_{\beta\cdot \tau}$.  \end{definition} 

This definition allows us to state a stronger version of Theorem~\ref{thm:premain}:

\begin{theorem}\label{thm:main}\label{thm:lift} Let $\beta$ be a coloured braid that acts on the marked disc labeled by $\tau$.  Then as maps from $\ws_\tau$ to $D$, the following  holds:

  \[ \pi_{\beta\cdot \tau}\circ \Phi(\beta)=\beta\circ \pi_\tau.\]
When $\beta\cdot \tau=\beta$, this implies that $\Phi(\beta)=\widetilde{\beta}$.
\end{theorem} 

The proof of Theorem~\ref{thm:main} is deferred until Section~\ref{sec:pfmain}.

\subsection{Examples}\label{sec:ex}

\begin{example}\label{ex:s3id} As a first example, we recover the familiar result that in a simple 3-fold cover of the disc branched over two points, the braid $\sigma^3$ lifts to the identity.    

Suppose that the initial labels are $(12)$ and $(23)$  and construct $(\ws_\tau, \Gamma_\tau)$ as shown in Figure~\ref{fig:s3}.  In this simple example, $\mathcal{G}(\mu)$ has three distinct graphical objects, and $\widetilde{\sigma}$ acts by cyclically permutating these.  The colouring is unnecessary, but included for convenience.

 \begin{figure}[h]
\begin{center}
\includegraphics[scale=.7]{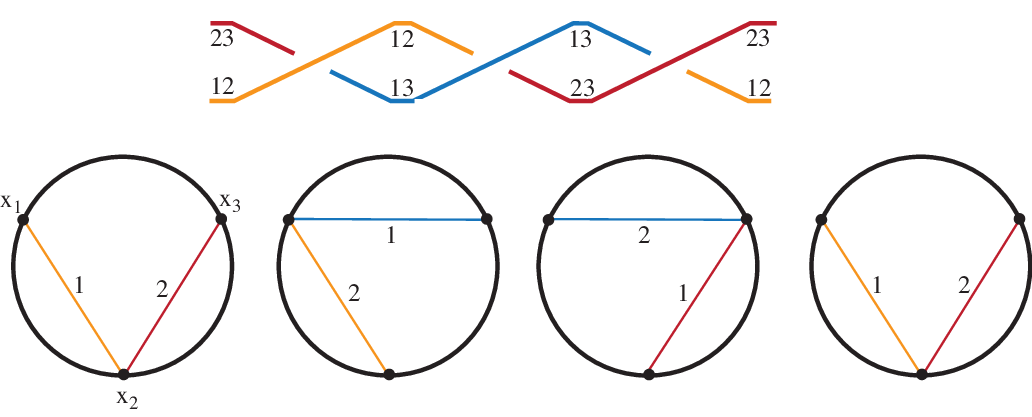}
\caption{ The morphisms $\widetilde{\sigma}$ permute the graphical objects in $(\mathcal{G}, \mathbf{x})$.   The braid $\sigma^3$ lifts to the identity. } 
\label{fig:s3}
\end{center}
\end{figure}

Since $\sigma^3\cdot \tau=\tau$, the braid $\sigma^3$  lifts to a mapping class of $\ws_\tau$.  This surface has a trivial mapping class group, so it's automatic that the lift is the identity mapping class, but this is confirmed by noting that the initial and final pictures are isotopic. 

This example establishes that $\sigma_i^3$ lifts to the identity whenever the two initial strands have distinct but non-disjoint labels, as the associated sequence of arcslides is supported (up to isotopy) in a neighbourhood of the initial $\widetilde{\gamma}_i$ and  $\widetilde{\gamma}_{i+1}$ curves.
\end{example}

\begin{example}\label{ex:ex2} Consider the braid introduced in Figure~\ref{fig:ex1}, again colouring the strands to identify the labeling transpositions more easily. As noted in the introduction, the top diagram factors as crossing between two strands with the same label followed by $\sigma_2^3$, which we have just shown lifts to the identity.  Since the strands that cross first have the same label, the braid should lift to a Dehn twist.  However, this factorisation is not easily seen from the lower diagram for the same braid,  which is reproduced below.  Setting $\tau= \big( (12), (12), (23)\big)$, we examine the action of $\phi(\beta)$ on $O_\tau$.  

\begin{figure}[h]
\begin{center}
\includegraphics[scale=.5]{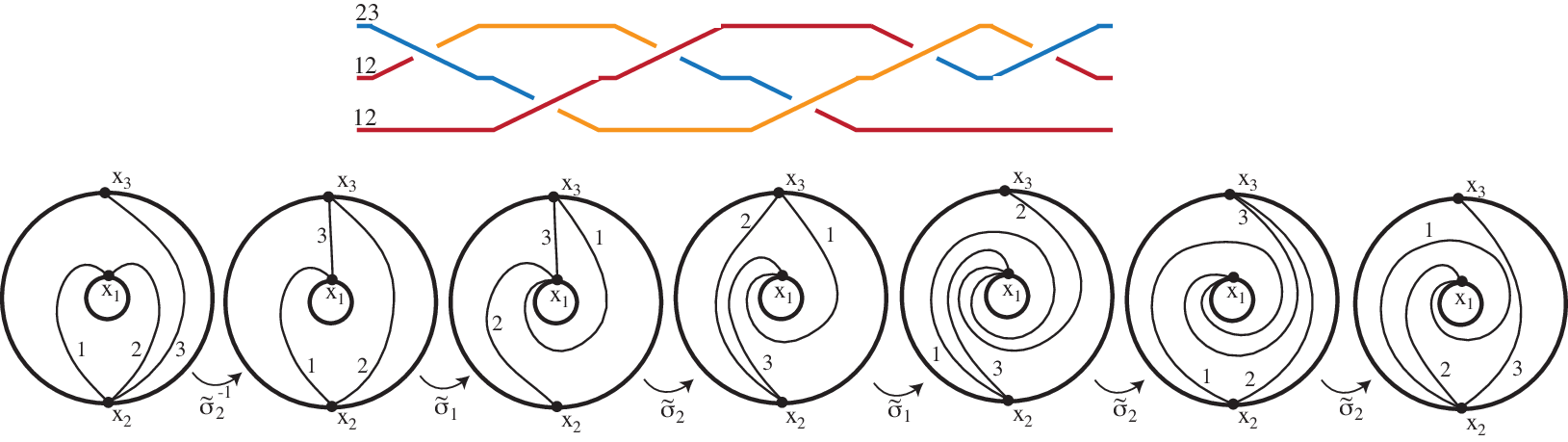}
\caption{ The braid above acts on the graphical object on the left as shown.} 
\label{fig:morph}
\end{center}
\end{figure}

In the braid diagram, each pair of crossing strands is labeled by a pair of distinct but not disjoint transpositions, so each $\widetilde{\sigma}_i$ is a single  arcslide.  After performing each one in turn, the right-most surface is $\phi(\beta) \cdot O_\tau$. It is easy to see that the final picture maps to the initial picture of  $O_\tau$ via a positive Dehn twist.  This yields the expected result that the lift of the generator $\sigma_i$ is a positive Dehn twist. 

\end{example}

\begin{example} Remark~\ref{rmk:build} identified $\ws_\tau$  as a surface neighbourhood of the rigid vertex graph $\widetilde{\Gamma}_\tau$.  This perspective makes it easy to construct braids  that lift to certain homeomorphisms.  For example, suppose that consecutive $\widetilde{\gamma}_i, \widetilde{\gamma}_{i+1}, \dots, \widetilde{\gamma}_{i+k}$ form a cycle when traversed in index order.  An indicative example is shown in the first picture of Figure~\ref{fig:ex5}. Then Dehn twists around this simple closed curve are achieved by lifting the following braid: 
\[ \sigma_i \sigma_{i+1} \dots \sigma_{i+k-2} \sigma_{i+k-1}^{\pm n}\sigma_{i+k-2}^{-1}\dots \sigma_i^{-1}.\]

Here, $\pm n$ is chosen based on the number and sign of the desired Dehn twists.

\begin{figure}[h]
\begin{center}
\includegraphics[scale=.65]{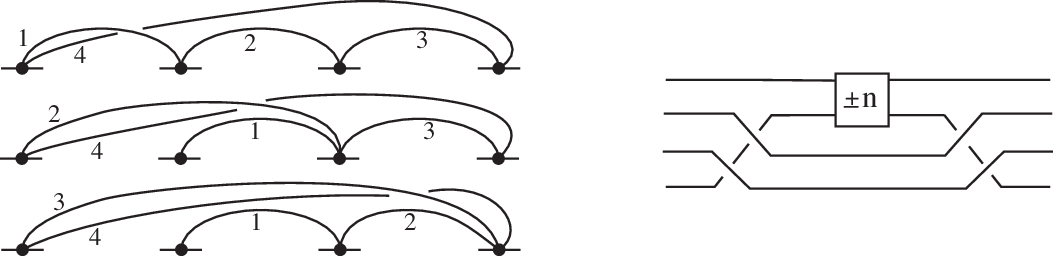}
\caption{Left: Sliding $\widetilde{\gamma}_1$ over successively indexed arcs produces a pair of adjacent strands with the same colour.   Right: Composing the left-hand preliminary braid with $\sigma_3^{\pm }$ and then the inverse of the preliminary braid lifts to a Dehn twist around the cycle formed by the initial $\widetilde{\gamma}_1,\widetilde{\gamma}_2, \widetilde{\gamma}_3$ and $\widetilde{\gamma}_4$. } 
\label{fig:ex5}
\end{center}
\end{figure}

In this case, the preliminary braid produces two $\widetilde{\Gamma}$ curves labeled by the same transposition, so the existence of Dehn twists is clear from the braid diagram even without computing the core of the twisting annulus.  In contrast, one may also construct coloured braids that ``hide'' the twists.  When $d\geq 3$, Etnyre-Casals have noted that every braid is represented by a braid diagram where no crossing strands have the same label \cite{CasEt}. (They state that every crossing may be assumed to have three distinct labels, but when the degree of the cover is greater than three, crossings between strands labeled with disjoint transpositions are also required.) The procedure for achieving this is quite concrete: whenever a crossing between two strands with the same label occurs, there is a strand with a different label that can be pulled through the crossing. For such diagrams, each $\sigma_i$ lifts to either an arcslide or a label swap, so at a cost of more crossings, the third case in Figure~\ref{fig:moves} may be omitted entirely.  

\end{example}

\subsection{Proof of Theorem~\ref{thm:main}}\label{sec:pfmain}

This section discusses the relationship between the two groupoids associated to a fixed total monodromy. 

\begin{theorem}\label{thm:homom2} The map $\Phi: \text{ColB}_{n,d}\rightarrow \mathcal{M}(\mu)$ is a groupoid homomorphism.
\end{theorem}

\begin{proof} 
Suppose that $\sigma_{i_2}\sigma_{i_1}$ is a coloured braid acting on the disc with initial labels $\tau$.  Recall that $\Phi(\sigma_{i_1})$ is a homeomorphism from $\ws_\tau$ to $\ws_{\sigma_{i_1}\cdot \tau}$ and $\Phi(\sigma_{i_2})$ is a homeomorphism from $\ws_{\sigma_{i_1}\cdot \tau}$ to $\ws_{\sigma_{i_2}\sigma_{i_1}\cdot \tau}$. 

Since $\Phi(\beta)$ is independent of the factorisation of $\beta$ as a product of generators, it suffices to show that for any Artin generators or their inverses, \begin{equation}\label{eq}\Phi(\sigma_{i_2}\sigma_{i_1})=\Phi(\sigma_{i_2})\circ\Phi(\sigma_{i_1}).\end{equation}
 
Furthermore, we must check that a trivial coloured braid maps to a trivial mapping class in $\mathcal{M}(\mu)$.  However, this follows immediately from Proposition~\ref{prop:braidrel}, which proved that trivial braids lift to identity morphisms in $\mathcal{G}(\mu)$.  Since $\phi(\beta) \cdot O_\tau$ is isotopic to $O_\tau$ for any trivial $\beta$, we see that $\Phi(\beta)$ is the identity morphism in $\mathcal{M}$.

A standard way to compare two homeomorphisms of a surface is to see that they act the same way on a sufficiently complicated set of curves.  In the present situation, we claim that the left- and right-hand sides of Equation~\ref{eq} act the same way on the lifts of the cut curves $\Delta$.  As with $\pi^{-1}(\Gamma)$, we discard components of $\pi^{-1}(\Delta)$ that do not meet at branch points, and we denote the remaining set of properly embedded arcs by $\widetilde{\Delta}$.  

Each of the three homeomorphisms in Equation~\ref{eq} is defined as a map sending a neighbourhood of one rigid vertex graph to a neighbourhood of another.  On the left-hand side, $\widetilde{\Gamma}_\tau$ is first altered by the $\mathcal{G}$-morphism $\phi(\sigma_{i_2}\sigma_{i_1})$ and then the surface neighbourhood of this new rigid vertex graph is mapped to $\widetilde{\Gamma}_{\sigma_{i_2}\sigma_{i_1}\cdot \tau}$ by $\Phi(\sigma_{i_2}\sigma_{i_1})$.  The intersection pattern between $\widetilde{\Gamma}_{\sigma_{i_2}\sigma_{i_1}\cdot \tau}$ and $\widetilde{\Delta}$ on $\ws_\tau$  is dictated by the specific arcslides and index swaps of $\phi(\sigma_{i_2}\sigma_{i_1})$, and this intersection pattern determines $\Phi(\sigma_{i_2}\sigma_{i_1})(\widetilde{\Delta})$.  

The right-hand side of Equation~\ref{eq} is a composition of two homeomorphisms, but the key observation is that the same arcslides and index swaps are applied to the rigid vertex graphs as in the first case.  There is an intermediate step that identifies a neighbourhood of $\phi(\sigma_{i_1})\cdot \widetilde{\Gamma}_\tau$ with $\widetilde{\Gamma}_{\sigma_{i_1}\cdot \tau}$ on $\widetilde{\Sigma}_{\sigma_{i_1}\cdot \tau}$, but this doesn't affect how the set of arcslides and index swaps associated to $\phi(\sigma_{i_2})$ change the intersection pattern with $\widetilde{\Delta}$.  Since  $\widetilde{\Gamma}_{\sigma_{i_2}\sigma_{i_1}\cdot \tau}$ has the same intersections with  $\Phi(\sigma_{i_2})\circ \Phi(\sigma_{i_1})(\widetilde{\Delta})$ and with $\Phi(\sigma_{i_2}\sigma_{i_1})(\widetilde{\Delta})$, it follows that the two images of $\widetilde{\Delta}$ are istotopic on $\ws_{\sigma_{i_1}\sigma_{i_2}\cdot \tau}$.  Hence, Equation~\ref{eq} holds.

\end{proof}

The map $\phi$ is defined on generators of the coloured braid group and extends to products by sequential application; $\Phi$ is defined so that the diagram relating decorated surfaces shown on  the left commutes. Forgetting the decorations yields the diagram on the right.

\[
\begin{minipage}{0.48\textwidth}
\centering
\begin{tikzcd}
(\widetilde{\Sigma}_\tau, \widetilde{\Gamma}_\tau) \arrow[r, "\phi(\beta)"] \arrow[d, "\pi_\tau"]
& (\widetilde{\Sigma}_\tau, \phi(\beta)\cdot \widetilde{\Gamma}_\tau) \arrow[r, "\Phi(\beta)"]
& (\widetilde{\Sigma}_{\beta\cdot\tau}, \widetilde{\Gamma}_{\beta\cdot \tau}) \arrow[d, "\pi_{\beta\cdot \tau}"] \\
(D^2, \tau, \Gamma) \arrow[rr, "\beta"]
&& (D^2, \beta\cdot\tau, \Gamma)
\end{tikzcd}
\end{minipage}
\hspace{1.5cm}
\begin{minipage}{0.48\textwidth}
\centering
\begin{tikzcd}
\widetilde{\Sigma}_\tau \arrow[r, "\Phi(\beta)"] \arrow[d, "\pi_\tau"]
& \widetilde{\Sigma}_{\beta\cdot\tau} \arrow[d, "\pi_{\beta\cdot \tau}"] \\
(D^2, \tau) \arrow[r, "\beta"]
& (D^2, \beta\cdot\tau)
\end{tikzcd}
\end{minipage}
\]

\begin{proof}[Proof of Theorem~\ref{thm:main}]
It remains to show that $\pi_{  \tau} \circ \Phi(\beta)=\beta \circ \pi_\tau$ for any liftable braid $\beta$ acting on the disc labeled with $\tau$. 
In light of Theorem~\ref{thm:homom2}, we will check the commutation equation \[\pi_{\sigma_i \cdot \tau} \circ \Phi(\sigma_i)=\sigma_i \circ \pi_\tau\] for all coloured Artin generators $\sigma_i$.  Since any coloured braid can be factored as a product of generators and the lift is independent of this factorisation, this suffices to prove that $\Phi(\beta)$ is the lift of $\beta$ in the sense of Definition~\ref{def:genlift}.

First, suppose that $\sigma_i$ is a half twist between marked points labeled with distinct but non-disjoint labels.  The covering surface splits into a subsurface preserved by the braid action and one that is altered, but the latter part is topologically a disc.  Up to isotopy, there is a unique way to map this disc in the initial covering surface to the corresponding disc in the final covering surface, so the fact that the coverings are equivalent implies that the diagram above commutes for such half twists.   

The case of disjoint labels is similar, but with the altered subsurface now a pair of discs.

Finally, suppose $\sigma_i$ exchanges marked points labeled with the same transposition.  In this case, $\sigma_i$ is itself liftable, so we consider the fixed surface $\ws_\tau$.   Following Definition~\ref{def:gomaps}, $\phi(\sigma_i)=\widetilde{\sigma_i}$ arcslides both ends of $\widetilde{\gamma_i}$ along $\widetilde{\gamma}_{i+1}$.  To recover a surface decorated with curves in the original isotopy class requires a Dehn twist along the core of $\widetilde{\gamma}_{i}\cup \widetilde{\gamma}_{i+1}$, as seen in the right-most picture in Figure~\ref{fig:moves}.  Thus $\Phi(\beta)$ is a right-handed Dehn twist along the simple closed curve covering  the line segment between $p_i$ and $p_{i+1}$, recovering the standard fact that a half twist lifts to the double cover as a Dehn twist.

\end{proof}

\section{Groupoids and CW-complexes}

There is an obvious way to build a $2$-dimensional CW complex from a groupoid, given sets of generating morphisms and relations: assign a $0$-cell to each object, a $1$-cell to each generating morphism, and attach a $2$-cell along each loop giving a relation among morphisms.    We say that complex built this way \textit{represents} the original groupoid. In this final section, we construct   CW-complexes representing the groupoids $\mathcal{G}(\mu)$ and $\mathcal{M}(\mu)$ and we note some of their topological properties.  Without loss of generality, assume that $\mu$ is the total monodromy of the distinguished label $\tau_0$.

In the case of the graphical groupoid $\mathcal{G}(\mu)$, the $0$-cells are graphical objects and we attach $2(k-1)$ $1$-cells corresponding to the $\widetilde{\sigma}_i^\pm$ to each $0$-cell.  Every path in this $1$-skeleton is a morphism.   This CW-complex has $|\mathcal{T}|$  connected components, since each morphism alters the decorations but not the underlying surface.  Choose the connected component containing the distinguished label $\tau_0$ and denote the it by $X^1_\mathcal{G}$.  

Next, consider the CW complex associated to the mapping class groupoid $\mathcal{M}(\mu)$.  The vertex set is indexed by $\mathcal{T}(\mu)$ and we may choose an analogous generating set for the morphisms.  

\begin{lemma} The homeomorphisms $\Phi(\sigma_i)$ generate the morphisms in $\mathcal{M}(\mu)$.  
\end{lemma}

\begin{proof} Since we have assumed that the degree of the cover is at least $3$, Theorem B in \cite{MAM} implies that the lifting map from the liftable braid group to the mapping class group of a fixed covering surface is surjective.  Here, we show that the lifting map from the coloured braid groupoid to the mapping class groupoid is also surjective.  

Let $\Psi$ be an arbitrary homeomorphism $\ws_\tau$ to $\ws_\rho$, for $\tau, \rho\in \mathcal{T}$.  Any coloured braid $\beta$ with initial labels $\tau$ and final labels $\rho$ lifts to a homeomorphism from  $\ws_\tau$ to $\ws_\rho$.  Then $\Psi=h_\rho \Phi(\beta) h_\tau$ for $h_*\in \text{Mod}(\ws_*)$.  Since each $h_*$ is the lift of some coloured braid, the result holds.  
\end{proof}

Let $X^1_\mathcal{M}$ denote the CW-complex built by assigning edges to the generating morphisms $\Phi(\sigma_i)$.  Here, as below, any path in $X^1_\mathcal{G}$ or $X^1_\mathcal{M}$ can be identified with a word in the Artin generators for the braid group, and once the initial $0$-cell is fixed, there is a well defined path associated to any braid. 

In both $\mathcal{G}(\mu)$ and $\mathcal{M}(\mu)$, the relations among morphisms  were defined by some global property, but  a more explicit description is required to attach $2$-cells to the CW-complexes.  We will introduce three types of relations, starting with the braid relations.  First, attach a $2$-cell to $X^1_*$ for each edge path labeled by a braid relation.

 Recall from Section~\ref{sec:brcov} that an arc $C$ on a marked and labeled disc is Type $i$ if $i$ is the minimal liftable power of the twist $\sigma_C$.   When $C$ is an elementary arc, its type can be read off easily from the labels on its endpoints.    Each elementary arc $C$ connecting branch points $i$ and $i+1$ labeled by disjoint transpositions is Type 2, and  $\sigma^2_C=\sigma^2_i$ lifts to the identity mapping class. Each elementary arc $C$ connecting branch points $i$ and $i+1$ labeled by distinct but non-disjoint transpositions is Type 3, and $\sigma^3_C=\sigma^3_i$ lifts to the identity mapping class. 

 This dictates the second class of relations: for each interval connecting branch points $i$ and $i+1$ labeled by disjoint transpositions, attach a $2$-cell along the loop $\sigma_i^2$ ; and  for each interval connecting branch points $i$ and $i+1$ labeled by distinct but non-disjoint transpositions, attach a $2$-cell along the loop $\sigma_i^3$.

 When $C$ is not an elementary arc, we may nevertheless identify its type: 
 
 \begin{lemma} The arc $C$ on the disc labeled by $\tau$ is a Type $i$ arc if there is some coloured braid $\beta$ with the property that $\beta\cdot C$ is an elementary Type $i$ arc on the disc labeled by $\beta\cdot \tau$.
\end{lemma} 

This criterion recovers Lemma 2.3 of \cite{MP}.  A braid $\beta$ that takes an arbitrary arc to an elementary one will not generally be liftable, but in the groupoid framework  it is natural to view a twist around a complicated arc as the conjugate of a twist around an elementary arc in a differently labeled disc.

\begin{proof} Conjugation preserves liftability, so the minimal liftable power of the twist $\sigma_C$ agrees with the minimal liftable power of the elementary twist $\beta^{-1}\sigma_C\beta$. 
\end{proof}

We have already attached a $2$-cell to the minimal liftable power of the twist around each elementary Type 2 and Type 3 arc, so for any Type 2 or Type 3 $C$ in $\mathcal{W}$, the minimal liftable power of $\sigma_C$ defines a contractible loop in each of $X_\mathcal{G}$ and $X_\mathcal{M}$.

The final set of relations reflect the complexity of surface mapping class groups, rather than being an artifact of our particular construction. Recall that for each equivalence class of coverings,  $\mathcal{W}$ is an explicit set of  arcs on the disc labeled with $\tau_0$ with the property that liftable powers of twists $\tau_C$ generate LB$_n$.

\begin{definition} Given a Type 1 arc $C\in \mathcal{W}$, let $\widetilde{\sigma}_C$ be the Dehn twist lifting $\sigma_C$.  
\end{definition}

Fix an elementary Type 1 arc $C$ on any labeled disc and consider the edge in $X^1_*$ associated to the twist $\sigma_C$.  Since the lift of $\widetilde{\sigma}_C$ is a homeomorphism, this edge is a loop in $X_\mathcal{M}$.  On the other hand, the lift to $\mathcal{G}$ induces a pair of arcslides which change the isotopy class of the decoration on the graphical object.  Hence the edge associated to $\sigma_C$ connects distinct vertices in $X_\mathcal{G}$.

Fix a class in $\text{Mod}(\ws_{\tau_0})$.  Any braid lifting to this class can be written as a product of liftable powers of $\sigma_C$ for $C \in \mathcal{C}(\tau_0)$, but powers of Type 2 and Type 3 arcs lift to the identity mapping class.  It follows that the Dehn twists $\widetilde{\sigma}_C$ generate $\text{Mod}(\ws_{\tau_0})$.  Since the mapping class group of a surface admits a finite presentation, there is a finite set of relations $R$ such that 
\[ \text{Mod}(\ws_{\tau_0})\cong \langle \{\widetilde{\sigma}_C\}_{C\in \mathcal{W}} \ | \ R\rangle.\]

Attach a $2$-cell to $X^1_*$ for each relation in $R$ in such a presentation.

Denote by $X_\mathcal{G}$ and $X_\mathcal{M}$ the $2$-dimensional CW complexes constructed from their $1$-skeleta by attaching 2-cells for braid relations; Type 2 and Type 3 twists; and mapping class relations $R$.   This is admittedly an indirect construction, but making it more explicit in any generality would require lengthy computation for minimal benefit.   Example~\ref{ex:complex} below shows some indicative $2$-cell attachments in a specific case.

\begin{theorem}\label{thm:cover} For $\ws$ the branched cover of the marked disc labeled by $\tau_0$, we have \[ \pi_1(X_\mathcal{M}, \tau_0)\cong\text{Mod}(\ws).\]  Furthermore, the CW-complex $X_\mathcal{G}$ is the universal cover of $X_\mathcal{M}$.  
\end{theorem} 

As a first step, we establish the relevance of these CW-complexes.

\begin{proposition}\label{prop:repr} The CW-complexes $X_\mathcal{G}$ and $X_\mathcal{M}$ represent the groupoids $\mathcal{G}(\mu)$ and $\mathcal{M}(\mu)$, respectively. 
\end{proposition}

 Proposition~\ref{prop:repr} has the following immediate corollary:

\begin{corollary}\label{cor} Then $\pi_1(X_\mathcal{M}, \tau_0)\cong \text{Mod}(\ws_{\tau_0})$.  
\end{corollary}

\begin{proof}[Proof of Proposition~\ref{prop:repr}] The $2$-cells in $X_*$ are associated to relations of three types, and we will show that these generate all the relations among morphisms in $\mathcal{G}$ and $\mathcal{M}$, as first defined  in Section~\ref{sec:groupoids}. 

Two morphisms in $\mathcal{G}(\mu)$ acting on the same graphical object are equivalent if and only if they produce the same graphical object.   Thus, every loop in $X_\mathcal{G}$ should be contractible.  Pick a loop of edges based at $\tau_0$ and consider the associated liftable braid word.  Because a 2-cell is attached for each braid group relation, this path is homotopic in $X_\mathcal{G}$ to a path factoring this braid as a product of liftable powers of WW generators.  Each factor that is a liftable power of a twist around at Type 2 or Type 3 arc $C$ corresponds to a path of the form $\beta^{-1}\sigma_j^i\beta$, where $\beta$ is a braid whose action on the marked disc takes $C$ to an elementary arc between the branch values $p_j$ and $p_{j+1}$.  A $2$-cell is attached to each $\sigma_j^i$ for $i=2,3$, so the sub-path corresponding to each of these factors is homotopic to the constant path at the basepoint.  We are left with a concatenation of loops each corresponding to a WW generator of the  mapping class group.  Since original path was a loop in $X_\mathcal{G}$, the composition of these Dehn twists represents the identity mapping class, and hence, some $2$-cells of the third type tile the loop, as desired.

An identical argument establishes the contractibilty of loops respresenting the identity mapping class in $X_\mathcal{M}$. In contrast to $X_\mathcal{G}$, there are loops in the this complex which are not contractible, but these correspond to non-trivial mapping classes.
\end{proof}

Finally,  we complete the proof of Theorem~\ref{thm:cover}.

\begin{proof}
The map $L: \mathcal{G}(\mu)\rightarrow \mathcal{T}$ that sends each graphical object to a set of labels on the marked disc can be interpreted as a map   $\mathcal{L}:X_\mathcal{G}\rightarrow X_\mathcal{M}$.   

Any $0$-cell of $X_\mathcal{G}$ is a graphical object $(\ws_{\tau}, \widetilde{\Gamma})$; define $\mathcal{L}(\ws_{\tau}, \widetilde{\Gamma})$ to be the vertex in $X_\mathcal{M}$ corresponding to $\ws_{L(\widetilde{\Gamma})}$.  The fibre over a $0$-cell in $X_\mathcal{M}$ labeled with $\tau$ is then the infinite set of $0$-cells associated to graphical objects $O$ such that $L(O)=\tau$.  

Extending $\mathcal{L}$ to edges is straightforward: send any edge labeled by $\widetilde{\sigma_i}$ in $X_{\mathcal{G}}$ to the edge $\Phi(\sigma_i)$ in $X_{\mathcal{M}}$.

Finally, recall that 2-cells were added to the two complexes using formally  identical recipes.  For each $2$-cell added to $X_\mathcal{M}$, the $\mathcal{L}$-preimage of the boundary is an infinite family of loops in $X_\mathcal{G}$ with a $2$-cell attached to each.  Define $\mathcal{L}$ on a $2$-cell in $X_\mathcal{G}$ to be the $2$-cell in $X_\mathcal{M}$ attached to the $\mathcal{L}$-image of its boundary.  The covering claim follows from this extension.

\end{proof}

\begin{example}\label{ex:complex}

This example gives a more explicit description of the complexes $X_\mathcal{G}$ and $X_\mathcal{M}$ for the equivalence class of branched covers of the disc associated to $\tau_0=(12), (23), (23), (23)$.  As seen in Figure~\ref{fig:complex3}, the branched cover is a one-holed torus.  The simplicity of this example is reflected in the size of the generating set $\mathcal{W}$; the only  twists around Type 1 curves in this set are $\sigma_2$ and $\sigma_3$.  These lift to the Dehn twists around the dotted and solid curves, respectively.

\begin{figure}[h]
\begin{center}
\includegraphics[scale=.6]{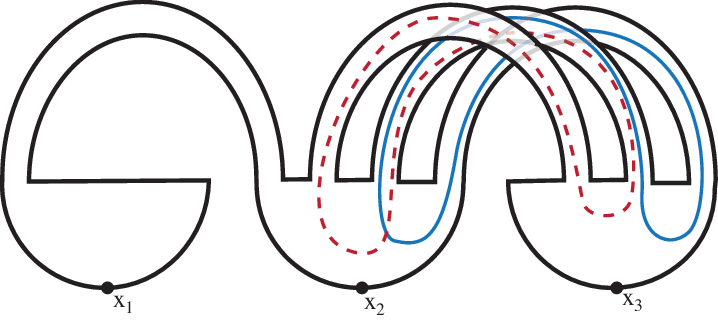}
\caption{ The mapping class group of the one-holed torus $\ws_{\tau_0}$ is generated by Dehn twists around two indicated curves.     } 
\label{fig:complex3}
\end{center}
\end{figure}

Figure~\ref{fig:complex2} indicates the kinds of $2$-cells attached in $X_\mathcal{M}$.  Since the the generator $\sigma_1$ exchanges branch values  with distinct but non-disjoint $\tau_0$ labels, a $2$-cell is attached along the path labeled by $\sigma_1^3$. A $2$-cell is attached along the loop labeled by the braid relation $\sigma_1\sigma_2\sigma_1\sigma_2^{-1}\sigma_1^{-1}\sigma_2^{-1}$.  In fact, the same reasoning dictates attaching a $2$-cell along the loop labeled by $\sigma_2\sigma_3\sigma_2\sigma_3^{-1}\sigma_2^{-1}\sigma_3^{-1}$, but we may alternatively view this as an example of a mapping class group relation, since Figure~\ref{fig:complex3} shows that $\Phi(\sigma_2)$ and $\Phi(\sigma_3)$ are Dehn twists along curves intersecting once in the cover.

\begin{figure}[h]
\begin{center}
\includegraphics[scale=.45]{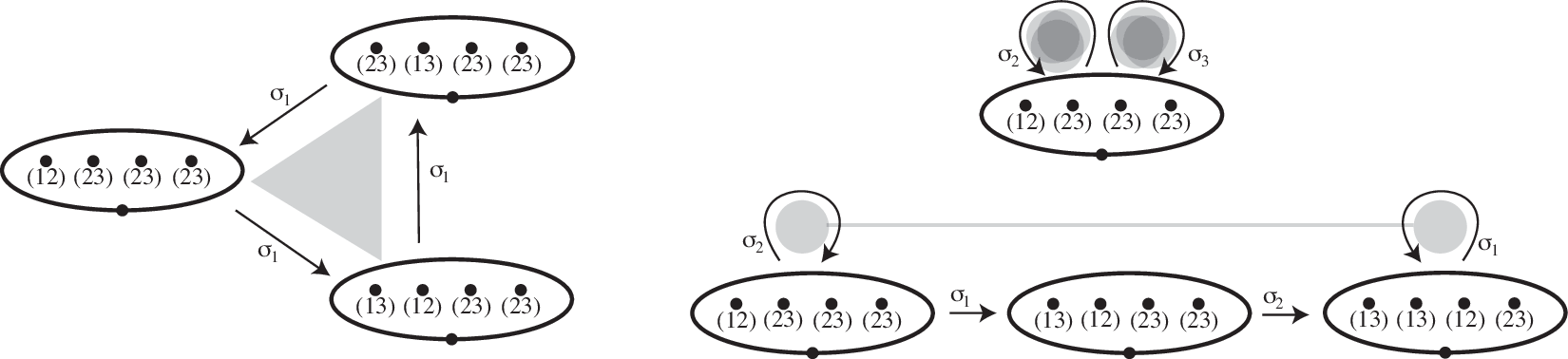}
\caption{ Left: attach a 3-cell along the edge path labeled by $\sigma_1^3$ in $X^1_\mathcal{M}$.  Right, bottom: attach a $2$-cell along $\sigma_1\sigma_2\sigma_1\sigma_2^{-1}\sigma_1^{-1}\sigma_2^{-1}$.  Right, top: attach a $2$-cell along $\sigma_2\sigma_3\sigma_2\sigma_3^{-1}\sigma_2^{-1}\sigma_3^{-1}$.   } 
\label{fig:complex2}
\end{center}
\end{figure}

\end{example}

We end with question.  The $2$-cells attached to $X^1_\mathcal{G}$ and $X^1_\mathcal{M}$ allow us to prove Theorem~\ref{thm:cover}, but resorting to relations  in the mapping class group is somewhat unsatisfying, as this backs away from the groupoid structure driving the rest of the paper.  It seems likely that there are other ways to characterise $2$-cells that produce representative CW-complexes.  In particular, we would like to understand which sequences of arcslides correspond to isotopic homeomorphisms without reference to the mapping class group of a fixed covering surface.

\begin{question} How else can $2$-cells be attached to $X^1_\mathcal{G}$ and $X^1_\mathcal{M}$ in order to build CW-complexes satisfying Theorem~\ref{thm:cover}?
\end{question}

\bibliographystyle{alpha}
\bibliography{lift}

\end{document}